\documentclass[12pt]{amsart}
\usepackage{fullpage,stmaryrd,tensor,url}
\newtheorem{theorem}{Theorem}[subsection]
\newtheorem{prop}[theorem]{Proposition}
\newtheorem{cor}[theorem]{Corollary}
\newtheorem{lemma}[theorem]{Lemma}

\theoremstyle{definition}
\newtheorem{hypothesis}[theorem]{Hypothesis}
\newtheorem{example}[theorem]{Example}
\newtheorem{defn}[theorem]{Definition}
\newtheorem{remark}[theorem]{Remark}
\newtheorem{construction}[theorem]{Construction}

\numberwithin{equation}{theorem}

\newcommand{\bA}{\mathbf{A}}
\newcommand{\bP}{\mathbf{P}}

\newcommand{\bv}{\mathbf{v}}
\newcommand{\bw}{\mathbf{w}}

\newcommand{\PP}{\mathbf{P}}

\newcommand{\CC}{\mathbb{C}}

\newcommand{\QQ}{\mathbb{Q}}
\newcommand{\RR}{\mathbb{R}}
\newcommand{\ZZ}{\mathbb{Z}}

\newcommand{\calE}{\mathcal{E}}
\newcommand{\calO}{\mathcal{O}}
\newcommand{\hyp}[2]{\tensor[_{#1}]F{_{#2}}}

\newcommand{\GitHub}{GitHub}
\newcommand{\magma}{\textsc{Magma}}

\newcommand{\sage}{\textsc{SageMath}}

\DeclareMathOperator{\an}{an}

\DeclareMathOperator{\GL}{GL}

\DeclareMathOperator{\ident}{id}
\DeclareMathOperator{\ord}{ord}
\DeclareMathOperator{\Real}{Re}
\DeclareMathOperator{\Spec}{Spec}

\begin{document}

\title{Frobenius structures on hypergeometric equations}
\author{Kiran S. Kedlaya}
\date{10 Dec 2021}
\thanks{The author was supported by NSF (grants DMS-1501214, DMS-1802161, DMS-2053473), UCSD (Warschawski Professorship),
and the IAS School of Mathematics (Visiting Professorship 2018--2019).
Additional funding/hospitality was provided by ICTP (September 2017), HRIM--Bonn (March 2018), KIAS--Seoul (August 2019), AIM (August 2019), EPSRC (grant EP/K034383/1), and the Simons Collaboration on Arithmetic Geometry, Number Theory, and Computation.}

\begin{abstract}
We give an exposition of Dwork's construction of Frobenius structures associated to generalized hypergeometric equations via the interpretation of the latter due to Gelfand--Kapranov--Zelevinsky in the language of $A$-hypergeometric systems. As a consequence, we extract some explicit formulas for the degeneration at $0$ in terms of the Morita $p$-adic gamma function.
\end{abstract}

\maketitle

\section{Introduction}

Hypergeometric differential equations, of arbitrary order, provide some key examples of Picard--Fuchs equations and of rigid local systems. As such, they admit $p$-adic analytic \emph{Frobenius structures} which interpolate the zeta functions associated to certain motives over finite fields.

The purpose of this note is to extract from Dwork's book \cite{dwork-gen} an explicit construction of Frobenius structures on hypergeometric equations (see Theorem~\ref{T:existence of Frobenius}), and in particular a formula for the residue at 0 (see Corollary~\ref{C:full rank Frob}), using $A$-hypergeometric systems in the sense of Gelfand--Kapranov--Zelevinsky \cite{gkz} (which we introduce in very little detail in 
\S\ref{sec:GKZ}).
We also give a brief indication of how this knowledge can be used as the basis for an efficient algorithm to compute the action of Frobenius on the (rational) crystalline realizations of hypergeometric motives, in the style of Lauder's \emph{deformation method} \cite{lauder}.
We have implemented this method in \sage{} \cite{github} and gotten good results in practice; however, some further analysis is needed on the tradeoff between rigor and efficiency caused by the
choice of working precision for certain power series and $p$-adic coefficients (see Remark~\ref{R:precision}).

\section{Generalities}

We first recall some general facts and definitions concerning ordinary differential equations, including the definition of a Frobenius structure.

\subsection{Ordinary differential equations}
\label{subsec:ode}

We first recall some standard concepts in order to set notation for them.

\begin{defn}
Let $D$ be a differential operator acting on a field $F$ of characteristic zero.
By a \emph{$D$-differential equation}, we will always mean a homogeneous linear differential equation in the
variable $y$
of the form
\begin{equation} \label{eq:diffeq}
D^n(y) + a_{n-1} D^{n-1}(y) + \cdots + a_0 y = 0
\end{equation}
with $a_0,\dots,a_{n-1} \in F$. For uniformity of notation, we set $a_n = 1$. 

By a \emph{$D$-differential system} of rank $n$, we will mean an equation in the variable $\bv$ (a column vector of length $n$) of the form
\begin{equation} \label{eq:matrix DE}
N \bv + D(\bv) = 0,
\end{equation}
where $N$ is an $n \times n$ matrix over $F$. This is the same structure as a \emph{connection} over $F$ whose underlying module  is equipped with a distinguished basis.
\end{defn}

\begin{remark} \label{R:DE to system}
Given the equation \eqref{eq:diffeq}, let $N$ be the companion matrix
\[
N = \begin{pmatrix}
0 & -1 & \cdots & 0 & 0 \\
0 & 0 & & 0 & 0 \\
\vdots & & \ddots & & \vdots \\
0 & 0 & &  0 & -1 \\
a_0 & a_1 & \cdots & a_{n-2} & a_{n-1}
\end{pmatrix};
\]
then the solutions of \eqref{eq:matrix DE} are precisely the vectors of the form
\[
\bv = \begin{pmatrix} y \\ D(y) \\ \vdots \\ D^{n-1}(y) \end{pmatrix}
\]
where $y$ is a solution of \eqref{eq:diffeq}.

Conversely, given the equation \eqref{eq:matrix DE}, note that for $U$ an invertible $n \times n$ matrix over $F$, the equation
\[
N_U \bw + D(\bw) = 0, \qquad N_U := U^{-1} N U + U^{-1} D(U)
\]
is equivalent to the original equation via the substitutions
\[
\bv \mapsto U \bw, \qquad \bw \mapsto U^{-1} \bv.
\]
The \emph{cyclic vector theorem} (see for example \cite[Theorem~5.4.2]{kedlaya-course})
then implies that for any choice of $N$, there exists some $U$ for which $N_U$ is a companion matrix.
However, 
there is typically no natural choice of $U$.
\end{remark}

\begin{defn}
Let $X$ be a locally ringed space over $\Spec \QQ$.
Let $\Omega$ be a coherent sheaf on $X$ equipped with a derivation $d: \calO_X \to \Omega$.
A \emph{connection} on $X$ (with respect to $d$) 
consists of a pair $(\calE, \nabla)$ in which $\calE$ is a vector bundle (locally free coherent sheaf) $\calE$ on $X$
and $\nabla: \calE \to \calE \otimes_{\calO_X} \Omega$ is an additive morphism satisfying the Leibniz rule with respect to $d$: for $U \subseteq X$ open, $f \in \Gamma(U, \calO)$, $\bv \in \Gamma(U, \calE)$,
we have 
\[
d(f\bv) = f \nabla(\bv) + \bv \otimes d(f).
\]
We also refer to such a pair as being a connection on $\calE$.
The elements of the kernel of $\nabla$ on $\calE(U)$ are called the \emph{horizontal sections} of $\calE$, or more precisely of $(\calE, \nabla)$, over $U$.

Given two connections $(\calE_1, \nabla_1)$, $(\calE_2, \nabla_2)$, 
the \emph{tensor product} is the connection $(\calE_1 \otimes_{\calO_X} \calE_2, \nabla)$ 
given by
\[
\nabla(f \bv \otimes \bw) = f \nabla_1(\bv) \otimes \bw + f \bv \otimes \nabla_2(\bw) + d(f) \otimes \bv \otimes \bw.
\]
Given a connection $(\calE, \nabla)$, the \emph{dual} is the unique connection  whose underlying bundle is the modulo-theoretic dual $\calE^\vee$ for which the canonical pairing $\calE \otimes \calE^\vee \to \calO_X$ is a morphism of connections.
\end{defn}

\begin{remark} \label{R:connection from matrix}
In the case where $X = \Spec F$, $\Omega = \calO_X$, $d = D$, and $\calE = \calO_X^{\oplus n}$,
any connection on $\calE$ has the form $\bv \mapsto N\bv + D(\bv)$ for some $n \times n$ matrix $N$ over $F$
(and conversely any such matrix defines a connection). The solutions of the equation \eqref{eq:diffeq}
then correspond to the horizontal sections of $\calE$ over $X$.
The dual connection (with the dual basis) corresponds to the matrix $-N^T$.
\end{remark}

\begin{defn}
Let $F\{D\}$ denote the \emph{Ore polynomial ring} in $D$; it is a noncommutative $F$-algebra whose underlying set coincides with that of $F[D]$, but whose multiplication is characterized by the identity
\[
Dx - xD = D(x) \qquad (x \in F).
\]
Then a connection on $\Spec F$ is the same as a left $F\{D\}$-module whose underlying $F$-vector space is identified with the set of length-$n$ column vectors over $F$,
with the action of $D$ given by $\bv \mapsto N\bv + D(\bv)$; passing from $N$ to $N_U$ amounts to changing basis on this vector space via the matrix $U$.

Given a $D$-differential system defined by a $D$-differential equation \eqref{eq:diffeq}, the dual of the corresponding connection is the left $F\{D\}$-module $F\{D\}/F\{D\}(D^n + a_{n-1} D^{n-1} + \cdots + a_0)$.
\end{defn}

\subsection{Regular singularities}
\label{sec:regular}

Throughout \S\ref{sec:regular}, let $K$ be a field of characteristic 0.

\begin{defn} \label{D:regular 1}
In the notation of \S\ref{subsec:ode}, take $F = K(z)$ to be equipped with the derivation $D = z\frac{d}{dz}$.
We then say that the equation \eqref{eq:diffeq} is \emph{regular} at $0$
if $\ord_0(a_i) \geq 0$ for $i=0,\dots,n-1$. We say that \eqref{eq:matrix DE} is \emph{regular} at $0$
if $\ord_0(N_{ij}) \geq 0$ for $i,j=1,\dots,n$.
\end{defn}

\begin{defn}
With notation as in Definition~\ref{D:regular 1}, fix an algebraic closure of $K$.
Define the \emph{local exponents} at 0 
of the equation \eqref{eq:matrix DE} to be the negations of the roots of the characteristic polynomial
of $N|_{z=0}$. By the classical theory of regular (Fuchsian) singularities, the images of the local exponents
under $\exp(2 \pi i \bullet)$ compute the eigenvalues of local monodromy around $z=0$.
Note that this only uses the values of the exponents modulo $\ZZ$; in fact it is only these residues that are intrinsic under meromorphic changes of coordinates, as one can make integral shifts using \emph{shearing transformations}
\cite[Proposition~7.3.10]{kedlaya-course}.
\end{defn}

\begin{defn} \label{D:regular}
Now in the notation of \S\ref{subsec:ode}, take $F = K(z)$
to be equipped with the derivation $D = \frac{d}{dz}$.
For $z_0 \in \bP^1_{K}$, the equation \eqref{eq:matrix DE} is \emph{regular} at $z_0$
if the entries of $N$ have at worst simple poles at $z=z_0$; for $z=0$, this is consistent with Definition~\ref{D:regular 1}.
The equation \eqref{eq:diffeq} is \emph{regular} at $z_0$ if
the corresponding matrix equation is; for $z_0 \in \bA^1_K$, this translates into the condition
\[
\ord_{z_0}(a_i) \geq i-n \qquad (i=0,\dots,n-1).
\]
\end{defn}

\subsection{Frobenius structures on differential equations}
\label{subsec:Frobenius structures}

\begin{hypothesis}
Throughout \S\ref{subsec:Frobenius structures}, fix a prime $p$.
Let $X$ be an open subspace of $\PP^1_{\QQ_p}$. Let $Z$ be the complement of $X$ in $\PP^1_{\QQ_p}$;
to simplify notation, we assume that $\{0, \infty\} \subseteq Z$.
\end{hypothesis} 

\begin{defn}
By a \emph{Frobenius lift}, we will mean a $\QQ_p$-linear map
$\sigma: \calO(X) \to \calO(X)$ such that $\sigma(z)-z^p \in p \ZZ_p[z]_{(p)}$. 
For instance, we may take $\sigma(z) = z^p$; we call this the \emph{standard Frobenius lift} (with respect to the coordinate $z$).
\end{defn}

\begin{defn}
Let $\PP^{1,\an}_{\QQ_p}$ be the analytification of $\PP^1_{\QQ_p}$ in the sense of rigid analytic geometry. (For the purposes of this discussion, we use Tate's model of $p$-adic analytic geometry; however any of the equivalent models of $p$-adic analytic geometry may be used instead, such as Berkovich spaces or Huber adic spaces.)

Let $(\calE, \nabla)$ be a connection on $X$.
We define a \emph{Frobenius structure} on $(\calE, \nabla)$ with respect to the Frobenius lift $\sigma$ as  an isomorphism $\sigma^* \calE \cong \calE$ of vector bundles with connection on some subspace $V$ of $\PP^{1,\an}_{\QQ_p}$ whose complement consists of a union of closed discs, each contained in the open unit disc around some point of $Z$.

More generally, for $(\calE', \nabla')$ another connection on $X$, 
we define a \emph{Frobenius intertwiner} from $(\calE, \nabla)$ to $(\calE', \nabla')$ with respect to the Frobenius lift $\sigma$ to be an isomorphism $\sigma^* \calE \cong \calE'$ of vector bundles with connection on some subspace $V$ as above.
\end{defn}

\begin{remark}
In the context of Remark~\ref{R:connection from matrix}, 
a Frobenius intertwiner corresponds to an invertible $n \times n$ matrix $\Phi$ with entries in the ring $\calO(V)$ satisfying
\begin{equation} \label{eq:commutation}
N' \Phi - c_\sigma  \sigma(N) + D(\Phi) = 0, \qquad
c_\sigma = \frac{\sigma(dz/z)}{dz/z} = \frac{D(\sigma(z))}{\sigma(z)}.
\end{equation}
The effect of changing basis by two invertible matrices $U,U'$ is to replace $\Phi$ with 
\[
\Phi_{U,U'} := U^{-1} \Phi \sigma(U'),
\]
which defines a Frobenius intertwiner from $N_U$ to $N'_{U'}$.
\end{remark}

\begin{remark} \label{R:normalization}
When a Frobenius intertwiner exists, one can always rescale it by an invertible elements of $\QQ_p$. In many cases, one can show that there can be at most one Frobenius structure up to rescaling (see Lemma~\ref{L:unique intertwiner} below); however, we will need some extra information in order to normalize for this scalar ambiguity.
\end{remark}

\begin{lemma} \label{L:unique intertwiner}
Let $(\calE, \nabla)$ and $(\calE', \nabla')$ be two connections on $X$ satisfying the following conditions.
\begin{enumerate}
\item[(a)]
The restriction of $(\calE, \nabla)$ to some open unit disc is trivial.
\item[(b)]
The points of $Z$ are pairwise noncongruent modulo $p$. 
\item[(c)]
At each $z \in Z$, $(\calE', \nabla')$ is regular with exponents in $\ZZ_{(p)}$.
\item[(d)]
The connection $(\calE', \nabla')$ is irreducible over $\QQ_p(z)$.
\end{enumerate}
Then up to $\QQ_p^\times$-scalar multiplication, there exists at most one Frobenius interwiner from $(\calE,\nabla)$ to $(\calE', \nabla')$.
\end{lemma}
\begin{proof}
By Baldassari's theorem on continuity of the radius of convergence of $p$-adic differential equations \cite{baldassarri}, 
condition (a) implies triviality of $(\calE, \nabla)$ also on the restriction to a generic open unit disc.
With this, we may apply \cite{dwork} to conclude.
\end{proof}

\begin{remark} \label{R:change of Frobenius lift}
While the definition of a Frobenius intertwiner was made in terms of the chosen Frobenius lift $\sigma$,
there is a certain independence from this choice: for any other Frobenius lift $\tilde{\sigma}$, there is a functorial way to transform Frobenius intertwiners with respect to $\sigma$ into Frobenius intertwiners with respect to $\tilde{\sigma}$ using the Taylor isomorphism. As we will mostly be concerned with Frobenius defined with respect to a fixed Frobenius lift $z \mapsto z^p$, we will not develop this point here; see for example \cite[\S 17.3]{kedlaya-course}.
\end{remark}

\begin{lemma} \label{L:regular singular with Frobenius}
Let $D_0$ denote the open unit disc around $0$, and suppose that $Z \cap D_0 = \{0\}$.
Let $(\calE, \nabla), (\calE', \nabla')$ be connections on $X$
which are regular at $0$ with exponents in $\QQ \cap \ZZ_{(p)}$. Suppose that 
there exists a Frobenius intertwiner $\Phi$ from $(\calE, \nabla)$ to $(\calE', \nabla')$  with respect to the standard Frobenius lift $\sigma$.
\begin{enumerate}
\item[(a)]
As multisets of $\QQ/\ZZ$, the local exponents of $(\calE', \nabla')$ at $0$ correspond to $p$ times the local exponents of $(\calE, \nabla)$.
\item[(b)]
On $D_0$, we have decompositions
\[
\calE \cong \bigoplus_{\lambda \in \ZZ_{(p)} \cap [0,1)} \calE_\lambda, \qquad \calE' \cong \bigoplus_{\mu \in \ZZ_{(p)} \cap [0,1)} \calE'_\mu
\]
of connections such that $\calE_\lambda$ (resp.\ $\calE'_\mu$) admits a basis on which $D = z \frac{d}{dz}$ acts by multiplication by $\lambda$ (resp. $\mu$) plus a nilpotent scalar matrix.
\item[(c)]
Any Frobenius structure $\Phi$ on $(\calE, \nabla)$ extends holomorphically to the punctured open unit disc around $0$ and meromorphically across $0$. More precisely, with bases as in (b), for $\lambda, \mu \in \ZZ_{(p)} \cap [0,1)$ with $p\lambda  \equiv \mu \pmod{\ZZ}$, $\Phi$ carries $\sigma^* \calE_\lambda$ into $\calE'_\mu$ and $t^{p\mu-\lambda} \Phi$ acts holomorphically on the chosen bases.
\end{enumerate}
\end{lemma}
\begin{proof}
Suppose first that the exponents at $0$ are all in $\ZZ$. 
In this case, (a) is trivial, (b) follows from \cite[Proposition~17.5.1]{kedlaya-course}, and (c) follows from (b)
by logic as in Remark~\ref{R:Frobenius solution} below.

To treat the general case, let $m$ be the least common denominator of the exponents; then pulling back along $z \mapsto z^m$ gives another pair of connections admitting a Frobenius intertwiner, to which we may apply the previous argument
to deduce the claim. Compare the proof of \cite[Lemma~2.3]{kedlaya-tuitman}.
\end{proof}

\begin{remark}
By making the substitution $z \mapsto z^{-1}$, we may immediately infer that Lemma~\ref{L:regular singular with Frobenius} holds with the point $0$ replaced by $\infty$. The same does not apply directly to other points of $\PP^1_{\QQ_p}$ because the relevant substitutions change the Frobenius lift; however, by Remark~\ref{R:change of Frobenius lift} we may still infer that Lemma~\ref{L:regular singular with Frobenius}(a) holds at any point of $\PP^1_{\QQ_p}$.
\end{remark}

We next introduce the idea that one can compute a Frobenius structure by solving a differential equation and imposing an initial condition.
\begin{remark} \label{R:Frobenius solution}
Assuming that a given pair of connections given by matrices $N,N'$ admits a Frobenius intertwiner $\Phi$ for the standard Frobenius lift $\sigma$, one can attempt to compute it by first finding formal solution matrices $U, U'$ of $N,N'$ at $0$, i.e., finding invertible matrices $U,U'$ over $\QQ_p \llbracket z \rrbracket$ for which $N_U$ and $N'_{U'}$ are scalar matrices. In the context of hypergeometric equations, we will even have explicit formulas for $U,U'$ in terms of hypergeometric series and their derivatives.

We may further ensure that $N_U, N'_{U'}$ are block diagonal matrices 
with blocks indexed by $\lambda \in \ZZ_{(p)} \cap [0,1)$,
in which each of the blocks $N_\lambda, N'_\lambda$ equals $\lambda$ plus a nilpotent matrix. In this case, $\Phi_U$ is itself a block permutation matrix with nonzero $(\lambda, \mu)$-block whenever $p\lambda \equiv \mu \pmod{\ZZ}$.
If we call this block $\Phi_\lambda$, as per \eqref{eq:commutation} we have
\[
N_\mu \Phi_\lambda + D(\Phi_\lambda) = p \Phi_\lambda N_\lambda.
\]
(Here we have replaced $\sigma(N_\lambda)$ with $N_\lambda$ because $N_\lambda$ has entries in $\QQ_p$, which are fixed by $\sigma$.) Since $N_\lambda$ and $N_\mu$ are scalar matrices, we may write $\Phi_\lambda = \sum_{n=-\infty}^\infty \Phi_n z^n$ and see that
\[
N_\mu \Phi_n + n \Phi_n = p \Phi_n N_\lambda;
\]
since $N_\lambda - \lambda$ and $N_\mu - \mu$ are nilpotent,
this implies that $\Phi_n = 0$ unless $\mu + n = p \lambda$; that is, $\Phi_\lambda$ equals $t^{p\lambda-\mu}$ times an invertible matrix over $\QQ_p$.
\end{remark}

\begin{remark} \label{R:bounded series}
Keeping notation as in Remark~\ref{R:Frobenius solution}, by writing $\Phi$ as $U \Phi_U \sigma(U')^{-1}$, we can express the entries of $\Phi$ as elements of $\QQ_p \llbracket z \rrbracket$. In order to be a Frobenius structure, these series have to also represent entries of $\calO(V)$ for some $V$; this in particular implies that the series in $\QQ_p \llbracket z \rrbracket$ we are considering have bounded coefficients, that is, they belong to the subring 
$\ZZ_p \llbracket z \rrbracket[p^{-1}]$ of $\QQ_p \llbracket z \rrbracket$. 

This containment generally does not hold ``by accident.'' For a typical differential equation, there is no choice of the scalar matrices $\Phi_{\lambda, 0} := t^{\mu - p\lambda} \Phi_\lambda$ for which this last containment holds; in this case, no Frobenius structure can exist. When a Frobenius structure does exist, typically the values of $\Phi_{\lambda,0}$ are uniquely determined, up to a joint scalar multiplication, by the fact that they give rise to entries of $F$ having bounded coefficients. This can be used as a mechanism for discovering the entries of $\Phi_{\lambda,0}$ empirically without any prior knowledge; see
\cite{vanstraten} for some examples of this and \cite{candelas-delaossa-vanstraten} for a more comprehensive treatment.

By contrast, in the case of hypergeometric equations, we will give a computable formula for the matrices $\Phi_{\lambda,0}$. (Since the entries are elements of $\QQ_p$ which are in general transcendental over $\QQ$, this means that for any fixed integer $N$, we can compute rational numbers which differ from the entries of $\Phi_{\lambda,0}$ by values in $p^N \ZZ_p$.)
\end{remark}

\begin{remark}
Keeping notation as in Remark~\ref{R:bounded series}, suppose that there exists a Frobenius structure $\Phi$ for which we have a computable formula for matrices $\Phi_{\lambda,0}$. The entries of $\Phi$ are elements of $\calO(V)$; this ring is a certain completion of $\calO(X)$ contained in the $p$-adic completion. We may thus represent the entries of $\Phi$ as sums of the form
\[
P(z) + \sum_{i=1}^\infty \frac{c_i}{Q(z)^i} \qquad (P(z) \in \QQ_p[z^{\pm}], c_i \in \QQ_p, \lim_{i \to \infty} c_i = 0)
\]
where $Q(z)$ is the monic polynomial with simple zeroes at $Z \setminus \{0, \infty\}$.
(In the case of hypergeometric equations, we will have $Q(z) = z-1$.)

In order to obtain a representation of $\Phi$ which is accurate to some prescribed $p$-adic accuracy, we need an effective bound on the decay rate of the $c_i$; this amounts to identifying a choice of the subspace $V$ and a bound on $\Phi$ over $V$. In the case where the points of $Z$ have pairwise distinct images under specialization,
this can be done by studying the effect of changing the Frobenius lift (Remark~\ref{R:change of Frobenius lift}).
\end{remark}

\section{Hypergeometric equations and the GKZ construction}
\label{sec:GKZ}

We now describe the generalized hypergeometric equation that we consider, the Gelfand--Kapranov--Zelevinsky construction of $A$-hypergeometric systems, and how the two are related.

\subsection{Hypergeometric differential equations}

\begin{defn}
Define the differential operator $D := z \frac{d}{dz}$ on the field $K(z)$ as in Definition~\ref{D:regular 1}.
The \emph{generalized hypergeometric equation} with parameters in $K$ given by
\[
\underline{\alpha}; \underline{\beta} = \alpha_1,\dots,\alpha_m; \beta_1,\dots,\beta_n 
\]
is the linear differential equation of the form
\begin{equation} \label{eq:hypergeom eq}
P(\underline{\alpha}; \underline{\beta})(y)= 0, \quad
P(\underline{\alpha}; \underline{\beta}) := z \prod_{i=1}^m (D + \alpha_i) - \prod_{j=1}^n (D+\beta_j-1).
\end{equation}
(We conflate this equation with the equivalent equation in terms of the operator $\frac{d}{dz}$, which is somewhat less compact to express.)
The case $m = n=2$ recovers the classical (Gaussian) hypergeometric equation.
We will primarily be interested in the case $K = \QQ$,
but in this section we treat the case $K = \CC$ following Beukers--Heckman \cite{beukers-heckman}.
\end{defn}

\begin{remark} \label{R:switch alpha beta}
Under the substitution $z \mapsto (-1)^{m-n} z^{-1}$, solutions of \eqref{eq:hypergeom eq} correspond to solutions of 
$P(\underline{\alpha}'; \underline{\beta}')(y) = 0$ for
\[
\underline{\alpha}'; \underline{\beta}' := 1-\beta_1,\dots,1-\beta_n; 1-\alpha_1,\dots,1-\alpha_m.
\]
\end{remark}

\begin{remark} \label{R:integer shift}
As in \cite[Proposition~2.3]{beukers-heckman}, one has
\begin{gather*}
(D+\delta-1)P(\underline{\alpha}; \underline{\beta}) =
P(\underline{\alpha},\delta; \underline{\beta}, \delta) \\
P(\underline{\alpha}; \underline{\beta})(D+\delta) =
P(\underline{\alpha},\delta; \underline{\beta}, \delta+1).
\end{gather*}
As per \cite[Corollary~2.4]{beukers-heckman}, it follows that for $i=1,\dots,m$ and $j=1,\dots,n$,
\begin{gather*}
P(\underline{\alpha}; \underline{\beta})(D+\alpha_i-1)
 = (D+\alpha_i-1) P(\alpha_1,\dots,\alpha_i-1,\dots,\alpha_m; \beta_1,\dots,\beta_n) \\
(D+\beta_j-1)P(\underline{\alpha}; \underline{\beta})
 = P(\alpha_1,\dots,\alpha_m;\beta_1,\dots,\beta_j-1,\dots,\beta_n) (D+\beta_j).
\end{gather*}
This has the consequence that for all practical purposes, the analysis of the hypergeometric equation is insensitive to integer shifts in the parameters. In particular, there is no real loss of generality in normalizing the parameters so that
\[
0 \leq \Real(\alpha_1) \leq \cdots \leq \Real(\alpha_m) < 1,
\qquad
0 \leq \Real(\beta_1) \leq \cdots \leq \Real(\beta_n) < 1;
\]
this will become convenient when we start manipulating series solutions of \eqref{eq:hypergeom eq}.
\end{remark}

\begin{remark}
For $n=1$,  \eqref{eq:hypergeom eq} becomes
\[
(z-1) D + (z-1)(1-\beta_1) + z (\alpha_1-\beta_1+1) = 0
\]
with formal solutions 
\[
y = c z^{1-\beta_1} (z-1)^{\alpha_1-\beta_1+1}. 
\]
\end{remark}

We next recall the explicit description of formal solutions of \eqref{eq:hypergeom eq} at $z=0$. The formal solutions at $z=\infty$ may be described similarly by interchanging the roles of the $\alpha$ and the $\beta$. The formal solutions at $z=1$ behave somewhat differently; see \cite[Proposition~2.8]{beukers-heckman}.
\begin{defn}
For $n$ a nonnegative integer, define the rising Pochhammer symbol
\[
(\alpha)_n := \alpha(\alpha+1)\cdots (\alpha+n-1).
\]
Define the \emph{Clausen--Thomae hypergeometric series}
\[
\hyp{m}{n-1}\left( \left. \genfrac{}{}{0pt}{}{\alpha_1,\dots,\alpha_m}{\beta_1,\dots,\beta_{n-1}} \right| z \right) := \sum_{k=0}^\infty \frac{(\alpha_1)_k \cdots (\alpha_m)_k}{(\beta_1)_k \cdots (\beta_{n-1})_k} \frac{z^k}{k!}.
\]
The case $m=n=3$ was first considered by Clausen \cite{clausen}; the general case was first considered by Thomae \cite{thomae}.
\end{defn}

\begin{prop} \label{P:series solution}
In \eqref{eq:hypergeom eq}, suppose that $\beta_n = 1$ (so that $(\beta_n)_k = k!$) and that no $\beta_i$ is a nonpositive integer (which is to say that $(\beta_i)_k \neq 0$ for all $k \geq 0$). Then
\[
\hyp{m}{n-1} \left( \left. \genfrac{}{}{0pt}{}{\alpha_1,\dots,\alpha_m}{\beta_1,\dots,\beta_{n-1}} \right| z \right)
\]
is a solution of \eqref{eq:hypergeom eq} in $\CC \llbracket z \rrbracket$.
\end{prop}
\begin{proof}
This may be seen by a direct calculation: applying the operator
$z(D+\alpha_1)\cdots(D+\alpha_m)$ to the given series yields
\[
\sum_{k=0}^\infty \frac{(\alpha_1)_{k+1} \cdots (\alpha_m)_{k+1}}{(\beta_1)_k \cdots (\beta_{n-1})_k} \frac{z^{k+1}}{k!}
\]
while applying $(D+\beta_1-1)\cdots(D+\beta_n-1) = (D+\beta_1-1)\cdots(D+\beta_{n-1}-1) D$ yields the equivalent expression
\[
\sum_{k=0}^\infty \frac{(\alpha_1)_k \cdots (\alpha_m)_k}{(\beta_1)_{k-1} \cdots (\beta_{n-1})_{k-1}} \frac{k z^k}{k!}. \qedhere
\]
\end{proof}

\begin{cor} \label{C:distinct parameters sols}
In \eqref{eq:hypergeom eq}, suppose that $m \leq n$ and that $\beta_1,\dots,\beta_n \in \QQ$ are pairwise distinct modulo $\ZZ$. Then the sums
\begin{equation} \label{eq:formal hypergeom}
z^{1-\beta_i} \hyp{m}{n-1} 
\left( \left. \genfrac{}{}{0pt}{}{\alpha_1-\beta_i+1, \dots, \alpha_m-\beta_i+1}{\beta_1-\beta_i+1, \dots, \widehat{\beta_i - \beta_i + 1},\dots,\beta_n - \beta_i + 1} \right| z \right)\qquad (i=1,\dots,n) 
\end{equation}
form a $\CC$-basis of the solutions of \eqref{eq:hypergeom eq}
in the Puiseux field $\bigcup_{l=1}^\infty \CC((z^{1/l}))$.
\end{cor}

By formally differentiating with respect to parameters, we see what happens when some of the $\beta$'s come together modulo $\ZZ$.
\begin{cor} \label{C:indistinct parameters sols}
In \eqref{eq:hypergeom eq}, suppose that no two of $\beta_1,\dots,\beta_n \in \QQ$ differ by a \emph{nonzero} integer
(e.g., because they all belong to $[0,1)$).
For each $\beta \in \{\beta_1,\dots,\beta_n\}$ occurring with multiplicity $\mu$, 
for $=1,\dots,\mu-1$, consider the sums
\begin{equation} \label{eq:formal hypergeom2}
z^{1-\beta} \sum_{i=0}^j
\frac{j! (\log z)^{j-i} }{(j-i)!}
 \sum_{k=0}^\infty [\epsilon^i] \left(\frac{(\alpha_1 - \beta + 1+\epsilon)_k \cdots (\alpha_m- \beta + 1+\epsilon)_k}{(\beta_1 - \beta + 1+\epsilon)_k \cdots (\beta_n - \beta + 1+\epsilon)_k}\right) z^k
\end{equation}
where $[\epsilon^i](*)$ means the coefficient of $\epsilon^i$ of the expansion of $*$ as a formal power series in $\epsilon$.
These then form a $\CC$-basis of the solutions of \eqref{eq:hypergeom eq}
in the ring $\bigcup_{m=1}^\infty \CC((z^{1/m}))[\log z]$.
\end{cor}
\begin{proof}
For $i=0$, Proposition~\ref{P:series solution} implies that \eqref{eq:formal hypergeom2} is a solution for $\epsilon=0$.
We obtain $\mu-1$ additional linearly independent solutions
by formally differentiating with respect to $-\beta$;
noting that the derivative of $z^{1-\beta}$ with respect to $-\beta$ is $(\log z) z^{1-\beta}$,
we obtain the claimed formula.
\end{proof}

\begin{cor} \label{C:formal solution matrix}
In \eqref{eq:hypergeom eq}, suppose that $\beta_1,\dots,\beta_n \in \QQ$ and
$0 \leq \beta_1 \leq \cdots \leq \beta_n < 1$.
Let $i_1 < \cdots < i_l$ be the sequence of indices $i \in \{1,\dots,n\}$ for which either $i=1$,
or $i >1$ and $\beta_{i-1} < \beta_i$. For $h=1,\dots,l$, let $\mu_h$ denote the multiplicity of $\beta_{i_h}$
(so that $\mu_j = i_{h+1} - i_h$ if $h<l$ and $n+1-i_h$ otherwise).
Define the series $f_1,\dots,f_n \in \CC \llbracket z \rrbracket$ by the following formula: for  $h=1,\dots,l$ and $j=0,\dots,\mu_h-1$,
\[
f_{i_h+j} := \frac{1}{j!} \sum_{k=0}^\infty [\epsilon^j] \left(\frac{(\alpha_1 - \beta + 1+\epsilon)_k \cdots (\alpha_m - \beta + 1+\epsilon)_k}{(\beta_1 - \beta + 1+\epsilon)_k \cdots (\beta_n - \beta + 1+\epsilon)_k}\right) z^k.
\]
Let $U$ be the matrix over $\CC \llbracket z \rrbracket$ given by the following formula:
for $h=1,\dots,l$; $i=1,\dots,n$; $j=0,\dots,\mu_h-1$,
\[
U_{i(i_h+j)} = \sum_{k=\max\{0,j-i+1\}}^j \frac{j! (i-1)!}{k! (j-k)! (i-1-j+k)!}
(D+1-\beta_{i_h}) ^{i-1-j+k}(f_{i_h+k}).
\]
Then $U$ is invertible and $N_U$ is a block matrix with block lengths $\mu_1,\dots,\mu_m$ in which
\[
(N_U)_{(i_h+i)(i_h+j)} = \begin{cases}
\beta_{i_h} - 1 & i=j \\
-j & j=i+1 \\
0 & \mbox{otherwise}
\end{cases}
 \qquad (h=1,\dots,m; 0 \leq i,j \leq \mu_h-1).
\]
\end{cor}
\begin{proof}
In the ring $\CC ((z)) [\log z]$, we may define the elements
$g_1,\dots,g_n$ so that for $h=1,\dots,m$, $j=0,\dots,\mu_h-1$, the series $g_{i_h+j}$
is given by \eqref{eq:formal hypergeom2} for $\beta =\beta_{i_h}$, omitting the factor of $z^{1-\beta}$.
Define the invertible $n \times n$ matrix $V$ over $\CC ((z)) [\log z]$ by setting
$V_{ij} = (D + 1-\beta_j)^{i-1}(g_j)$; then $N_V$ is the diagonal matrix with entries
$\beta_1-1,\dots,\beta_n-1$.

By construction, we have
\[
g_{i_h+j} = \sum_{k=0}^j \binom{j}{k} (\log z)^k f_{i_h+j-k} \qquad (j=0,\dots,\mu_h-1);
\]
consequently, for $i=1,\dots,n$ we have
\begin{align*}
(D+1-\beta_{i_h})^{i-1}(g_{i_h+j}) &= \sum_{l=0}^j \binom{j}{l}  (D+1-\beta_{i_h})^{i-1}((\log z)^l f_{i_h+j-l})\\
&= \sum_{k=0}^j \binom{j}{k}
(\log z)^{j-k}  \sum_{l=j-k}^{j} * (D+1-\beta_{i_h})^{i-1-l+j-k}(f_{i_h+j-l}), \\
*& = \frac{k! (i-1)!}{(j-l)! (l-j+k)! (i-1-l+j-k)!}.
\end{align*}
That is, we have $V = UW$ where $W$ is the block matrix with block lengths $\mu_1,\dots,\mu_m$ in which
\[
W_{(i_h+i)(i_h+j)} = \binom{j}{i} (\log z)^{j-i} \qquad (0 \leq i,j \leq \mu_h-1);
\]
it follows that $N_U = W N_V W^{-1} + W D(W^{-1})$.
Since each block of $N_V$ is a scalar matrix, we have $W N_V W^{-1} = N_V$;
meanwhile, an elementary computation shows that the $h$-th block of $W D(W^{-1})$ is nilpotent with
superdiagonal entries $-1,-2,\dots,-\mu_h+1$.
\end{proof}

We recall the local structure of the singularities of \eqref{eq:hypergeom eq} in the case $m=n$.
\begin{prop} \label{P:local exponents}
For $m = n$, the equation \eqref{eq:hypergeom eq} is regular with singularities at $0,1,\infty$ having local exponents
as follows:
\begin{align*}
z=0: & \qquad 1-\beta_1,\dots,1-\beta_n \\
z=\infty: & \qquad \alpha_1, \dots, \alpha_n \\
z=1: & \qquad 0,\dots,n-2, \gamma, \qquad \gamma := \sum_{i=1}^n \beta_i - \sum_{i=1}^n \alpha_i.
\end{align*}
\end{prop}
\begin{proof}
See \cite[\S 2]{beukers-heckman}.
\end{proof}

Although we will not use this overtly,
for context we recall the explicit description of the monodromy representation of \eqref{eq:hypergeom eq}.
\begin{prop} \label{P:monodromy rep}
Suppose that $m=n$ and that $\alpha_i - \beta_j \notin \ZZ$ for $i,j=1,\dots,n$. 
\begin{enumerate}
\item[(a)]
Put $a_i := \exp(2\pi i \alpha_i), b_i := \exp(2 \pi i \beta_i)$ and define the polynomials
\[
\prod_{i=1}^n (T - a_i) = T^n + A_1 T^{n-1} + \cdots + A_n, \quad
\prod_{i=1}^n (T - b_i) = T^n + B_1 T^{n-1} + \cdots + B_n.
\]
Then in a suitable basis (see Remark~\ref{R:suitable basis}), the local monodromy operators \eqref{eq:hypergeom eq} may taken to be
\begin{gather*}
h_0 := B^{-1}, \qquad h_1 := A^{-1} B, \qquad h_\infty := A \in \GL_n(\CC), \\
A := \begin{pmatrix}
0 & 0 & \cdots & 0 & -A_n \\
1 & 0 & \cdots & 0 & -A_{n-1} \\
\vdots & & \ddots & & \vdots \\
0 & 0 & \cdots & 1 & -A_1
\end{pmatrix}, \qquad
B := \begin{pmatrix}
0 & 0 & \cdots & 0 & -B_n \\
1 & 0 & \cdots & 0 & -B_{n-1} \\
\vdots & & \ddots & & \vdots \\
0 & 0 & \cdots & 1 & -B_1
\end{pmatrix}.
\end{gather*}
\item[(b)]
The representation described in (a) is irreducible.
\item[(c)]
The matrix $h_1$ is a \emph{complex reflection} with special eigenvalue $c := \exp(2 \pi i \gamma)$, meaning that $h_1 - 1$ has rank $1$. 
\end{enumerate}
\end{prop}
\begin{proof}
Part (a) is a theorem of Levelt \cite[Theorem~3.5]{beukers-heckman}. Parts (b) and (c) are immediate corollaries; see \cite[Proposition~3.3]{beukers-heckman} for (b) and \cite[Proposition~2.10]{beukers-heckman} for (c).
\end{proof}

\begin{remark} \label{R:suitable basis}
In Proposition~\ref{P:monodromy rep}, if one further assumes that the $\alpha_i$ and $\beta_j$ are all distinct mod $\ZZ$, one can make the choice of a ``suitable basis'' quite explicit in terms of the local solutions
given by Corollary~\ref{C:distinct parameters sols}. This was originally shown by Golyshev--Mellit
\cite{golyshev-mellit}.
\end{remark}

\begin{remark}
In case $m \neq n$, the local structure of the singularities of \eqref{eq:hypergeom eq} is rather different; to simplify notation, we assume that $m < n$. In this case, \eqref{eq:hypergeom eq} is of order $n$
and its local monodromy at $0$ is as described above; however, we no longer have a singularity at $z=1$, and the singularity at $z=\infty$ is now irregular. 
This can be understood in terms of \emph{confluence}, where the regular singularities at 1 and $\infty$ have coalesced into an irregular singularity upon degeneration of one of the parameters. To make this more explicit, consider the one-parameter family of hypergeometric equations
\[
P(\alpha_1,\dots,\alpha_m, 1/t,\dots,1/t; \beta_1,\dots,\beta_n)
\]
indexed by a parameter $t$. This is equivalent via the substitution $z \mapsto t^{n-m} z$ to the equation
\[
z \prod_{i=1}^m (D+\alpha_i) \prod_{i=m+1}^n (t D + 1) - \prod_{j=1}^n (D+\beta_j-1)
\]
with a regular singularity at $z = t^{m-n}$. Taking the limit as $t \to 0$ yields the operator $P(\alpha_1,\dots,\alpha_m;\beta_1,\dots,\beta_n)$.
\end{remark}

\subsection{The GKZ interpretation}

In preparation for adopting the point of view of Dwork \cite{dwork-gen},
we recall the description of the hypergeometric equation \eqref{eq:hypergeom eq} in terms of a GKZ (Gelfand--Kapranov--Zelevinsky) $A$-hypergeometric system, following \cite{dwork-loeser} (see also \cite{adolphson}, \cite[\S 1.4]{cattani}, and
\cite[\S 2]{castano}). 

We begin by rewriting the hypergeometric equation to simplify the dependence on the parameters $\underline{\alpha}, \underline{\beta}$ at the expense of replacing the original series with a function of multiple variables. (Warning: the use of the letter $\Phi$ here has nothing to do with the Frobenius intertwiners discussed in \S\ref{subsec:Frobenius structures}.)

\begin{lemma} \label{L:HG to GKZ}
Consider a function $\Phi(\underline{x}, \underline{y})$ of indeterminates
$\underline{x}= x_1,\dots,x_m$ and $\underline{y} =y_1,\dots,y_n$. (For the moment, 
we leave it unspecified what sort of function we have in mind.)
\begin{enumerate}
\item[(a)]
The function $\Phi$ is annihilated by the operators
\begin{equation} \label{eq:euler}
x_j \frac{\partial}{\partial x_j} + y_k \frac{\partial}{\partial y_k} + \alpha_j -\beta_k + 1 \qquad (j=1,\dots,m;k=1,\dots,n)
\end{equation}
if and only if there exists a univariate function $f(z)$ such that
\begin{equation} \label{eq:Phi in terms of f}
\Phi(\underline{x}, \underline{y}) = x_1^{-\alpha_1} \cdots x_m^{-\alpha_m} y_1^{\beta_1-1} \cdots y_n^{\beta_n-1}
f((-1)^m x_1^{-1} \cdots x_m^{-1} y_1 \cdots y_n).
\end{equation}
\item[(b)]
For $\Phi,f$ satisfying \eqref{eq:Phi in terms of f}, $\Phi$ is annihilated by the operator
\begin{equation} \label{eq:toric}
\prod_{j=1}^m \frac{\partial}{\partial x_j} - \prod_{j=1}^n \frac{\partial}{\partial y_j}.
\end{equation}
if and only if $f$ is a solution of the hypergeometric equation \eqref{eq:hypergeom eq}.
\end{enumerate}
\end{lemma}
\begin{proof}
For $\Phi$ as in \eqref{eq:Phi in terms of f} and $z = (-1)^m x_1^{-1} \cdots x_m^{-1} y_1 \cdots y_n$,
we have
\begin{align}
\label{eq:toric alpha}
x_j \frac{\partial}{\partial x_j}(\Phi)(\underline{x}, \underline{y}) &= ((-D-\alpha_j)(f))(z) \\
\label{eq:toric beta}
y_j \frac{\partial}{\partial y_j}(\Phi)(\underline{x}, \underline{y}) &= ((D+\beta_j-1)(f))(z).
\end{align}
In particular, any such $\Phi$ satisfies \eqref{eq:euler}. Conversely, to check that any $\Phi$ satisfying \eqref{eq:euler} satisfies \eqref{eq:Phi in terms of f} for some $f$, we may formally reduce to the case
where $\alpha_i = 0, \beta_i = 1$ for all $i$. In this case, \eqref{eq:euler} implies that $\Phi$ remains constant under any substitution of the form 
\[
x_j \mapsto c x_j, \qquad y_k \mapsto c y_k
\]
(for some $j,k$, with the other variables left unchanged);
consequently, \eqref{eq:Phi in terms of f} holds for
\[
f(z) := \Phi(1,\dots,1,(-1)^m z).
\]
This proves (a).

For $z$ as above, the operator \eqref{eq:toric} may be rewritten as
\[
y_1^{-1} \cdots y_n^{-1} \left(
z \prod_{j=1}^m  \left(-x_j \frac{\partial}{\partial x_j} \right) - \prod_{j=1}^n \left(y_j \frac{\partial}{\partial y_j}\right)
\right).
\]
This makes it clear that from \eqref{eq:toric alpha}, \eqref{eq:toric beta}, we immediately deduce (b).
\end{proof}

\begin{cor} \label{C:HG to GKZ}
Suppose that $\beta_1,\dots,\beta_n$ are pairwise distinct modulo $\ZZ$.
In terms of the indeterminates $\underline{x}, \underline{y} = x_1,\dots,x_m,y_1,\dots,y_n$, 
for $i=1,\dots,n$ define (formally)
\begin{align*}
f_i(z) &:= z^{1-\beta_i} \hyp{m}{n-1} 
\left( \left. \genfrac{}{}{0pt}{}{\alpha_1-\beta_i+1, \dots, \alpha_m-\beta_i+1}{\beta_1-\beta_i+1, \dots, \widehat{\beta_i - \beta_i + 1},\dots,\beta_n - \beta_i + 1} \right| z \right) \\
\Phi_i(\underline{x}, \underline{y}) &:= x_1^{-\alpha_1} \cdots x_m^{-\alpha_m} y_1^{\beta_1-1} \cdots y_n^{\beta_n-1}
f_i((-1)^m x_1^{-1} \cdots x_m^{-1} y_1 \cdots y_n).
\end{align*}
Then $\Phi_1, \dots, \Phi_n$ are all annihilated by the operators \eqref{eq:euler} and \eqref{eq:toric}.
\end{cor}
\begin{proof}
Combine Lemma~\ref{L:HG to GKZ} with
Corollary~\ref{C:distinct parameters sols}.
\end{proof}

\begin{defn} \label{D:GKZ}
For $m$ a positive integer, let $W_m := \CC \langle x_1,\dots,x_m, \partial_1,\dots,\partial_m \rangle$
denote the \emph{Weyl algebra}, i.e., the quotient of the noncommutative polynomial algebra in $x_1,\dots,x_m,\partial_1,\dots,\partial_m$ by the two-sided ideal generated by
\[
x_i x_j - x_j x_i,\, \partial_i \partial_j - \partial_j \partial_i,\, \partial_i x_i - x_i \partial_i - 1 \qquad (i,j=1,\dots,m).
\]
We write $\theta_i$ as shorthand for $x_i \partial_i$.

For $d$ a nonnegative integer, let $A$ be a $d \times m$ matrix over $\ZZ$.
(In the notation of \cite[\S 2]{adolphson}, our $d$ is $n$ therein, our $m$ is $N$ therein, 
and the columns of $A$ correspond to the lattice points therein.)
The \emph{toric ideal} associated to $A$ is the ideal
\[
I_A = \{\partial^u - \partial^v: u,v \in \ZZ_{\geq 0}^{m}, Au = Av\} \subseteq \CC[\partial_1,\dots,\partial_m].
\]
For $\delta \in \CC^d$ a column vector, for $i=1,\dots,d$ we may define an \emph{Euler operator}
\[
A_{i1}\theta_1 + \cdots + A_{im} \theta_m - \delta_i \in W_m.
\]
The \emph{GKZ ideal} (or \emph{hypergeometric ideal}) defined by $A$ and $\delta$ is the left ideal $J_{A,\delta}$ of $W_m$ generated by $I_A$ and the Euler operators.
\end{defn}

\begin{example} \label{exa:HG to GKZ}
Define the $(m+n-1) \times (m+n)$ matrix $A$ over $\ZZ$ by the block expression
\[
A = \begin{pmatrix} I_m & 0 & 1 \\
0 & -I_{n-1} & 1
\end{pmatrix};
\]
the toric ideal is generated by
$\partial_1 \cdots \partial_m - \partial_{m+1} \cdots \partial_{m+n}$.
Let $\delta \in \CC^{m+n}$ be the column vector 
\[
(\alpha_1 - \beta_n + 1,\dots,\alpha_m - \beta_n + 1,
\beta_1 - \beta_n, \dots, \beta_{n-1} - \beta_n);
\]
the Euler operators then have the form
\begin{gather*}
\theta_j + \theta_{m+n} + \alpha_j - \beta_n + 1 \quad (j=1,\dots,m) \\
-\theta_{m+j} + \theta_{m+n} + \beta_j - \beta_n \quad (j=1,\dots,n-1).
\end{gather*}
By Lemma~\ref{L:HG to GKZ}, the formula
\begin{equation} \label{eq:hg to gkz}
\Phi(x_1,\dots,x_{m+n}) = x_1^{-\alpha_1} \cdots x_m^{-\alpha_m} x_{m+1}^{\beta_1-1}\cdots x_{m+n}^{\beta_n-1}
f((-1)^m x_1^{-1} \cdots x_m^{-1} x_{m+1} \cdots x_{m+n})
\end{equation}
defines a bijection between the functions $f(z)$ satisfying \eqref{eq:hypergeom eq} and the functions $\Phi(x_1,\dots,x_{m+n})$ annihilated by $J_{A,\delta}$.
\end{example}

It will be useful to also have a symmetric variant of Example~\ref{exa:HG to GKZ}.

\begin{example} \label{exa:HG to GKZ2}
Define an $mn\times (m+n)$ matrix $A$ over $\ZZ$, using the index set $\{1,\dots,m\} \times \{1,\dots,n\}$ in place of 
$\{1,\dots,mn\}$, by
\[
A_{(i_1,i_2)j} = \begin{cases} 
1 & j \in \{i_1, m+i_2\} \\
0 & \mbox{otherwise.}
\end{cases}
\]
and a column vector $\delta \in \CC^{mn}$ by
\[
\delta_{(i_1, i_2)} = \alpha_{i_1} - \beta_{i_2} + 1.
\]
The Euler operators then have the form
\[
\theta_{i_1} + \theta_{i_2} + \alpha_{i_1} - \beta_{i_2} + 1 \quad (i_1=1,\dots,m; i_2=1,\dots,n).
\]
This GKZ system is isomorphic to the previous one, in a sense to be made explicit in
\S\ref{sec:morphisms}.
\end{example}

\begin{remark} \label{R:convolution product}
Let $d',m'$ be two more positive integers, let $A$ be a $d' \times m'$ matrix over $\ZZ$, and let $\delta' \in \CC^{d'}$. We then have a canonical isomorphism of $\CC$-vector spaces
\[
W_m/J_{A,\delta} \otimes_{\CC} W_{m'}/J_{A',\delta'} \cong W_{m+m'}/J_{A\oplus A',\delta \oplus \delta'}
\]
which promotes to an isomorphism of left $W_{m+m'}$-modules if we identify the variables of $W_{m'}$ with 
the variables $x_{m+1},\dots,x_{m+m'},\partial_{m+1},\dots,\partial_{m+m'}$ of $W_{m+m'}$.
\end{remark}

\begin{remark} \label{R:drop last column}
In Example~\ref{exa:HG to GKZ}, if we drop the last column, the toric ideal becomes the zero ideal. In this case, the functions annihilated by $J_{A,\delta}$ are just the constant multiples
of $x_1^{-\alpha_1+\beta_n-1} \cdots x_m^{-\alpha_m+\beta_n-1} x_{m+1}^{\beta_1-\beta_n}\cdots x_{m+n-1}^{\beta_{n-1} - \beta_n}$; this can be viewed as an instance of the product construction described in Remark~\ref{R:convolution product}.
\end{remark}

\begin{remark}
A comment related to Remark~\ref{R:drop last column} is that the definition of a GKZ system in Example~\ref{exa:HG to GKZ} is insensitive to an overall translation
\[
\alpha_i \mapsto \alpha_i + c, \qquad \beta_i \mapsto \beta_i + c;
\]
the value of $c$ only appears in the comparison with the hypergeometric equation in 
\eqref{eq:hg to gkz}
(and specifically in the exponents of the leading powers).
\end{remark}

\subsection{Dwork's exponential module}

Returning to the general GKZ setup, we now introduce Dwork's construction of the exponential module
(compare \cite[\S 4]{dwork-coh}).

\begin{defn} \label{D:exponential module}
Retain notation as in Definition~\ref{D:GKZ}.
Let $R_A$ be the $\CC$-subalgebra of $\CC[X_1^{\pm},\dots,X_d^{\pm}]$ generated by the monomials
$X^{(j)} := X_1^{A_{1j}} \cdots X_d^{A_{dj}}$ for $j=1,\dots,m$.
Define also $R_A[x] := R_A[x_1,\dots,x_m]$.
Define the element
\[
g_{A} := \lambda \sum_{j=1}^m x_j X^{(j)} \in R_A[x].
\]
(In the original construction one takes $\lambda= 1$;
since we can absorb $\lambda$ by rescaling $x_j$ there is no extra generality in varying $\lambda$,
but this will be convenient for the construction of Frobenius structures.)

There are obvious ``natural'' actions of the derivations
\[
\partial_1,\dots,\partial_m, \qquad \Theta_1, \dots, \Theta_d := X_1 \frac{\partial}{\partial X_1}, \dots, X_d \frac{\partial}{\partial X_d}
\]
on $R_A[x]$ (but not $\frac{\partial}{\partial X_i}$ in general).
Define the twisted operators
\begin{align*}
\partial_{A,j} &:= \partial_j + \partial_j(g_{A}) = \partial_j + \lambda x_j X^{(j)} \\
D_{A,\delta,i} &:= \Theta_i + \Theta_i(g_{A}) + \delta_i = X_i \frac{\partial}{\partial X_i} + \delta_i + \lambda \sum_{j=1}^m A_{ij} x_j X^{(j)}.
\end{align*}
We give $R_A[x]$ the structure of a left $W_m$-module by specifying that
$\partial_j$ acts via $\partial_{A,j}$.
\end{defn}

\begin{remark}
In the setting of Example~\ref{exa:HG to GKZ}, we have
\begin{align*}
g_A &= \lambda (x_1 X_1 + \cdots + x_m X_m + x_{m+1} X_{m+1}^{-1} + \cdots \\
&\qquad \qquad + x_{m+n-1} X_{m+n-1}^{-1} + x_{m+n} X_1 \cdots X_{m+n-1}) \\
D_{A,\delta,i} &= X_i \frac{\partial}{\partial X_i} + \lambda x_i X_i + \lambda x_{m+n} X_1 \cdots X_{m+n-1} + \alpha_i - \beta_n + 1 \\
D_{A,\delta,i} &= X_i \frac{\partial}{\partial X_i} - \lambda x_i X_i^{-1} + \lambda x_{m+n} X_1 \cdots X_{m+n-1} + \beta_{i-m} - \beta_n.
\end{align*}
where the second and third equations are for $i=1,\dots,m$ and $i=m+1,\dots,m+n-1$ respectively.
\end{remark}

\begin{lemma} \label{L:GKZ transform}
The formula 
\[
x_1 \mapsto x_1, \dots, x_m \mapsto x_m, \partial_1 \mapsto X^{(1)}, \dots, \partial_m \mapsto X^{(m)}
\]
defines a surjective homomorphism $\phi: W_m \to R_A[x]$ of left $W_m$-modules
(for the exotic module structure on $R_A[x]$ from Definition~\ref{D:exponential module})
which induces the following isomorphisms of left $W_m$-modules:
\begin{align*}
W_m/W_m I_A &\cong R_A[x] \\
W_m/J_{A,\delta} &\cong R_A[x] / \sum_{i=1}^d D_{A,\delta,i} R_A[x].
\end{align*}
\end{lemma}
\begin{proof}
See \cite[Theorem~4.4]{adolphson}. (Compare also \cite[Theorem~6.8]{dwork-loeser} and \cite[Corollary~11.1.3]{dwork-gen}.)
\end{proof}

\begin{remark}
Even beyond the setting of Example~\ref{exa:HG to GKZ}, one can give a good ``toric'' description of 
$W_m/J_{A,\delta}$. As this is not necessary for our purposes, we defer to \cite{adolphson} for details.
\end{remark}

\subsection{Morphisms of $A$-hypergeometric systems}
\label{sec:morphisms}

\begin{defn}
Let $A'$ be a $d' \times m'$ matrix over $\ZZ$ and let $\delta' \in \CC^{d'}$ be a column vector.
By a \emph{morphism} from the GKZ hypergeometric system with parameters $(A,\delta)$
to the GKZ hypergeometric system with parameters $(A', \delta')$,
we will mean a homomorphism $\psi: R_A[x]\to R_{A'}[x]$ of $\CC$-modules which induces a homomorphism
\[
\overline{\psi}: R_A[x]/\sum_{i=1}^d D_{A,\delta,i} R_A[x] \to R_{A'}[x]/ \sum_{i'=1}^{d'} D_{A',\delta',i'} R_{A'}[x].
\]
In order to make this meaningful, we must also have some compatibility with the $\partial_j$; we will describe this on a case-by-case basis. 
\end{defn}

\begin{construction} \label{const:linear change}
Let $B$ be a $d' \times d$ matrix over $\ZZ$ and let $B'$ be a $d \times d'$ matrix over $\ZZ$ satisfying
\[
BA = A', \qquad B'A' = A, \qquad B\delta = \delta', \qquad B' \delta' = \delta.
\]
Consider the $\CC$-linear ring homomorphisms $\psi: R_A[x] \to R_{A'}[x]$, $\psi': R_{A'}[x] \to R_A[x]$ of $\CC$-modules given by
\begin{align*}
\psi: & \quad x_j \mapsto x_j, \qquad X_i \mapsto \prod_{i'=1}^{d'} X_{i'}^{B_{i'i}}, \\
\psi': & \quad x_j \mapsto x_j, \qquad X_{i'} \mapsto \prod_{i=1}^d X_i^{B'_{ii'}}.
\end{align*}
These satisfy the following identities:
\begin{gather*}
\psi' \circ \psi = \ident_{R_A[x]}, \qquad \psi \circ \psi' = \ident_{R_{A'}[x]}, \\
\psi(g_A) = g_{A'}, \qquad \psi'(g_{A'}) = A, \\
D_{A',\delta',i'} \circ \psi = \sum_i B_{i'i} \psi \circ D_{A,\delta,i}, \\
D_{A,\delta,i} \circ \psi' = \sum_{i'} B'_{ii'} \psi' \circ D_{A',\delta',i'}.
\end{gather*}
Consequently, $\psi$ and $\psi'$ define morphisms $(A,\delta) \to (A', \delta')$, $(A',\delta') \to (A, \delta)$
which are inverses of each other and manifestly commute with $\partial_1,\dots,\partial_m$.
\end{construction}

\begin{example} \label{exa:permutations}
In Example~\ref{exa:HG to GKZ2}, we have obvious isomorphisms as in Construction~\ref{const:linear change}
corresponding to the permutations of 
$\alpha_1,\dots,\alpha_n$ and of $\beta_1,\dots,\beta_n$; however, these are not automorphisms because they change $\delta$. We may similarly construct an isomorphism effecting the interchange of parameters from
Remark~\ref{R:switch alpha beta}.
\end{example}

\begin{example} \label{exa:GKZ1 to GKZ2}
We construct an isomorphism, in the sense of Construction~\ref{const:linear change},
between the minimal GKZ system corresponding to a hypergeometric equation (Example~\ref{exa:HG to GKZ})
and the more symmetric version (Example~\ref{exa:HG to GKZ2}). This uses the matrices
\[
B_{(i_1,i_2)i} = \begin{cases}
1 & i = i_1 \\
-1 & i = m+i_2 \\
0 & \mbox{otherwise.}
\end{cases}
\qquad
B'_{i(i_1,i_2)} = \begin{cases}
1 & (i_1, i_2) = (i, n) \\
1 & (i_1, i_2) = (i-m, n) \\
-1 & (i_1, i_2) = (i-m, i-m) \\
0 & \mbox{otherwise.}
\end{cases}
\]
\end{example}

\begin{construction} \label{const:integer translate}
Let $T \in \ZZ^d$ be a vector in the column span of $A$
and put $A' := A$, $\delta' := \delta - T$. 
Let $\psi: R_A[x] \to R_{A'}[x]$ be the map given by multiplication by $X_1^{T_1} \cdots X_d^{T_d}$;
it satisfies
\[
D_{A',\delta',i} \circ \psi = \psi \circ D_{A,\delta,i} \qquad (i=1,\dots,d)
\]
and therefore defines a morphism $(A,\delta) \to (A', \delta')$ which manifestly commutes with 
$\partial_1,\dots,\partial_m$.
\end{construction}

We now consider some cases where the interaction with $\partial_1,\dots,\partial_m$ is a bit more subtle.

\begin{construction} \label{const:specialization}
Let $A'$ be the $d \times (m-1)$ matrix obtained from $A$ by omitting the last column,
and put $\delta' := \delta$.
The ring homomorphism $\psi: R_A[x] \to R_{A'}[x]$ specializing $x_m$ to 0 then satisfies
\[
D_{A',\delta,i} \circ \psi = \psi \circ D_{A, \delta,i} \qquad (i=1,\dots,d);
\]
consequently, it defines a morphism $(A, \delta) \to (A', \delta')$ which commutes with the operators $\partial_1,\dots,\partial_{m-1}$. This does not extend to $\partial_m$ because no such operator has been defined on $R_{A'}[x]$.
\end{construction}

\begin{construction} \label{const:algebraic Frobenius}
Put $A' = A$, $\delta' := p\delta$,
and consider the morphism $\varphi: R_A[x] \to R_{A}[x]$ given by the substitution $x_j \mapsto x_j^p, X_i \mapsto X_i^p$.
If we define
\[
h := \lambda \sum_{j=1}^m ( x_j X^{(j)} - (x_j X^{(j)})^p),
\]
then
\begin{gather*}
\left( D_{A,p\delta,i} - \Theta_i(h)
\right) \circ \varphi =  p \varphi \circ D_{A,\delta,i}
 \qquad (i=1,\dots,d) \\
(x_j \partial_{A,j} - x_j \partial_j(h)) \circ \varphi = p\varphi \circ (x_j \partial_{A,j}) \qquad (j=1,\dots,m).
\end{gather*}
Formally, this means that $\exp(h) \varphi$ is a morphism which defines a Frobenius intertwiner
(because of the factor of $p$ in the second relation).
In the $p$-adic context, this becomes not merely formal because of the convergence properties of the Dwork exponential series (for a suitable choice of $\lambda$).
\end{construction}

\begin{remark}
Somewhat tangentially to our current discussion, we note that one could also make the Frobenius intertwiner nonformal by working over a base ring 
equipped with a topology in which $\lambda$, $x_j - 1$, and $X_i - 1$ are small enough to make the series $\exp(h)$ convergent.
This hints towards a potential connection with $q$-de Rham cohomology in the sense of Scholze \cite{scholze-q}
and prismatic cohomology in the sense of Bhatt--Scholze \cite{bhatt-scholze-prisms}.
\end{remark}

\section{Hypergeometric Frobenius intertwiners}

We now give our interpretation of Dwork's construction of Frobenius intertwiners for hypergeometric equations, based on morphisms of $A$-hypergeometric systems.

\subsection{Existence of Frobenius intertwiners}

\begin{defn} \label{D:Dwork exponential}
Fix a choice of $\pi$ in an algebraic closure of $\QQ_p$ satisfying $\pi^{p-1} = -p$.
Define the \emph{Dwork exponential series} to be the series
\[
E_\pi(t) := \sum_{j=0}^\infty c_j t^j = \exp(\pi(t-t^p));
\]
it has radius of convergence $p^{(p-1)/p^2} > 1$ \cite[\S VII.2.4]{robert}.
\end{defn}

\begin{theorem}[Dwork] \label{T:existence of Frobenius}
Let $\underline{\alpha}; \underline{\beta}$ and $\underline{\alpha}', \underline{\beta}'$ be two sequences in $\ZZ_p$
such that $p \underline{\alpha}; p \underline{\beta}$ are congruent modulo $\ZZ$ to some permutations of $\underline{\alpha}', \underline{\beta}'$. Then over $\QQ_p(\pi)$, there exists a Frobenius intertwiner between the connections corresponding to $P(\underline{\alpha},\underline{\beta})$ and $P(\underline{\alpha}', \underline{\beta}')$.
\end{theorem}
\begin{proof}
We construct the desired intertwiner as follows.
\begin{itemize}
\item
Take $A, \delta$ as in Example~\ref{exa:HG to GKZ2}, then apply Construction~\ref{const:algebraic Frobenius}
(taking $\lambda$ there to be our chosen $\pi$)
to replace $\underline{\alpha}; \underline{\beta}$
with $p \underline{\alpha}; p \underline{\beta}$.
\item
Use Construction~\ref{const:integer translate} to replace $p \underline{\alpha}; p \underline{\beta}$
with a permutation of $\underline{\alpha}'; \underline{\beta}'$.
\item
Use Example~\ref{exa:permutations} to undo the permutation of $\underline{\alpha}'; \underline{\beta}'$.
\end{itemize}
Note that the convergence property of the Dwork exponential is needed in the first step.
\end{proof}

\begin{remark}
In Theorem~\ref{T:existence of Frobenius}, if $m=n$ and $\alpha_i - \beta_j \notin \ZZ$ for all $i,j$, then 
we may combine Lemma~\ref{L:unique intertwiner} and Proposition~\ref{P:monodromy rep} to deduce that the Frobenius intertwiner is unique up to scalar multiplication. On the other hand, we can resolve the ambiguity completely by observing that the construction given by Theorem~\ref{T:existence of Frobenius} has the following properties.
\begin{enumerate}
\item[(a)]
In case $\underline{\alpha} = \underline{\alpha}', \underline{\beta} = \underline{\beta}'$, the Frobenius intertwiner is the identity.
\item[(b)]
The construction of the Frobenius intertwiner is compatible (in a natural sense which we decline to notate) with permutations of each of $\underline{\alpha}, \underline{\beta}, \underline{\alpha}', \underline{\beta}'$.
\item[(c)]
Suppose that
\begin{equation} \label{eq:compare alpha beta with prime}
\alpha'_i = p\alpha_i + \mu_i, \qquad \beta'_j = p\alpha_j + \nu_j \qquad (\mu_i, \nu_j \in \ZZ).
\end{equation}
Then the restriction of the Frobenius intertwiner to any fixed point of $X$ varies $p$-adically continuously as we vary
$\underline{\alpha}, \underline{\beta}$ while maintaining \eqref{eq:compare alpha beta with prime} and fixing $\mu_i,\nu_j$.
\end{enumerate}
\end{remark}

\begin{remark}
For $n=2$, an alternate construction of the Frobenius intertwiner has been given by Salinier \cite{salinier} using rigidity;
this has been generalized to all $n$ by Vargas Montoya \cite{vargasmontoya}. While this approach is technically simpler than
Dwork's method, the latter is more useful for our ultimate aim of making explicit computations.
\end{remark}

\subsection{Gamma factors and the Dwork exponential series}

In order to make use of Construction~\ref{const:algebraic Frobenius},
we recall the description due to Dwork\footnote{The attribution is predicated on the fact that \cite{boyarsky} was written by Dwork under the pseudonym Maurizio Boyarsky \cite[p. 341, first sidebar]{dwork-notices}.} \cite{boyarsky}, \cite[\S 1]{dwork-boyarsky} of the relationship between the Morita $p$-adic gamma function and Gauss sums provided by the Gross--Koblitz formula \cite{gross-koblitz}. See Remark~\ref{R:CM motive} for the geometric interpretation of this.

\begin{defn}
Recall (or see \cite[\S VII.1.1]{robert}) that there exists a unique continuous function $\Gamma_p: \ZZ_p \to \ZZ_p^\times$ characterized by the properties 
\begin{align}
\Gamma_p(0) &= 1 \label{eq:gamma1}\\
\frac{\Gamma_p(x+1)}{\Gamma_p(x)} &= \begin{cases} 
-x & x \notin p\ZZ_p \\
-1 & x \in \ZZ_p.
\end{cases} \label{eq:gamma2}
\end{align}
This function is the \emph{Morita $p$-adic gamma function}.
\end{defn}

\begin{defn} \label{D:Dwork definition of gamma}
For $a,b \in \ZZ_{(p)} \setminus \ZZ$ with $pb - a = \mu \in \ZZ$,
Dwork defines the symbol $\gamma_p(a, b) \in \QQ_p(\pi)$ by the formula
\[
\gamma_p(a,b) = \sum_{i \in \ZZ} c_{pi + \mu}(b)_i/(-\pi)^i.
\]
Equivalently, writing
\[
\psi(f)(x) = \frac{1}{p} \sum_{z^p = x} f(x),
\]
we have
\begin{equation} \label{eq:one-dim Frob}
\psi(x^{a-pb} E_\pi(x)) \equiv \gamma_p(a,b) \mod{\left( x \frac{d}{dx} + b + \pi x\right) \QQ_p(\pi) \llbracket x \rrbracket}.
\end{equation}
For fixed $\mu \in \ZZ$, using the series representation we may extend $\gamma_p(pb-\mu, b)$ to a continuous function of $b \in \ZZ_p$; note that $\gamma_p(0,0) = 1$.
For $s,t \in \ZZ$, we have the functional equation \cite[(1.7)]{dwork-boyarsky}
\begin{equation} \label{eq:boyarsky funceq}
\gamma_p(a+s,b+t) = \gamma_p(a,b) (-\pi)^{t-s} \frac{(a)_s}{(b)_t}.
\end{equation}
\end{defn}

\begin{theorem}[Dwork] \label{T:DGK formula}
For $a,b \in \ZZ_{(p)}$ with $pb - a =\mu \in \{0,\dots,p-1\}$, we have
\[
\gamma_p(a,b) = \pi^{\mu} \Gamma_p(a).
\]
\end{theorem}
\begin{proof}
Using the above discussion, one checks that $\gamma(a,b)/\pi^\mu$ satisfies the defining properties
\eqref{eq:gamma1}, \eqref{eq:gamma2} of $\Gamma_p(a)$; this proves the claim.
\end{proof}

As indicated in \cite{boyarsky},
Theorem~\ref{T:DGK formula} can be viewed as an equivalent form of the Gross--Koblitz formula for Gauss sums
\cite{gross-koblitz}. In other words, we immediately compute the Frobenius intertwiners for hypergeometric equations of order $1$.

\begin{cor} \label{C:rank 1 Frob}
Let $\{x\} := x - \lfloor x \rfloor$ denote the fractional part of $x$.
In the case
\[
m=n=1, \qquad \alpha_1, \alpha'_1, \beta_1, \beta'_1 \in [0,1), \qquad \alpha_1 \neq \beta_1,
\]
for
\[
\mu = p(\alpha_1 - \beta_1) - (\alpha'_1 - \beta'_1) \in \ZZ,
\]
the Frobenius interwiner of Theorem~\ref{T:existence of Frobenius} is given by 
multiplication by 
\begin{align*}
\gamma(\alpha'_1 - \beta'_1 + 1, \alpha_1 - \beta_1 + 1) &:= \pi^\mu \Gamma_p(\{\alpha'_1 - \beta'_1\}) \times  \\
&\qquad
\left(  \left( \begin{cases} \frac{1}{\alpha_1 - \beta_1} & \alpha_1 > \beta_1 \\ p & \alpha_1 < \beta_1 \end{cases} \right) \times \left(\begin{cases} \alpha'_1 - \beta'_1 & \alpha'_1 > \beta'_1 \\ -1 & \alpha'_1 < \beta'_1 \end{cases}\right) \right).
\end{align*}
\end{cor}
\begin{proof}
We first make some auxiliary calculations in order to prepare for the use of Theorem~\ref{T:DGK formula}.
Note that
\[
p\alpha_1 - \alpha'_1, p \beta_1 - \beta'_1 \in \ZZ \cap (-1, p) = \{0,\dots,p-1\}
\]
and so
\[
\mu  = (p\alpha_1 - \alpha'_1) - (p\beta_1 - \beta'_1) \in \{1-p, \dots, p-1\}.
\]
If $\alpha_1 > \beta_1$, then we also have $p(\alpha_1 - \beta_1) \in (0,p)$, $\alpha'_1 - \beta'_1 \in (-1,1)$ and so
\begin{equation} \label{eq:normalized range1}
\mu \in \{1-p,\dots,p-1\} \cap (-1, p+1) = \{0,\dots,p-1\}.
\end{equation}
Similarly, if $\alpha_1 < \beta_1$, then $p(\alpha_1 - \beta_1) \in (-p,0)$, $\alpha'_1 - \beta'_1 \in (-1,1)$ and so
\[
\mu \in \{1-p,\dots,p-1\} \cap (-p-1, 1) = \{1-p,\dots,0\}
\]
and
\begin{equation} \label{eq:normalized range2}
p(\alpha_1 - \beta_1+1) - (\alpha'_1 - \beta'_1+1) = (p-1) + \mu \in \{0,\dots,p-1\}.
\end{equation}
In particular, $\mu$ is either zero or has the same sign as $\alpha_1 - \beta_1$.

By \eqref{eq:one-dim Frob} and \eqref{eq:boyarsky funceq}, the Frobenius intertwiner is given by multiplication by 
\[
\gamma(\alpha'_1 - \beta'_1 + 1, \alpha_1 - \beta_1 + 1) = 
\gamma(\alpha'_1 - \beta'_1, \alpha_1 - \beta_1) \frac{\alpha'_1 - \beta'_1}{\alpha_1 - \beta_1}.
\]
In case $\alpha_1 < \beta_1$, we apply Theorem~\ref{T:DGK formula} and \eqref{eq:normalized range2} to write
\[
\gamma(\alpha'_1 - \beta'_1 + 1, \alpha_1 - \beta_1 + 1) = \pi^{(p-1) + \mu} \Gamma_p(\alpha'_1 - \beta'_1 + 1).
\]
In case $\alpha_1 > \beta_1$, we may apply Theorem~\ref{T:DGK formula} and \eqref{eq:normalized range1} to write
\[
\gamma(\alpha'_1 - \beta'_1, \alpha_1 - \beta_1) = \pi^{\mu} \Gamma_p(\alpha'_1 - \beta'_1).
\]
We can thus write the intertwiner as
\begin{equation} \label{eq:intertwiner1}
 \begin{cases}
\pi^{\mu} \frac{\alpha'_1 - \beta'_1}{\alpha_1 - \beta_1} \Gamma_p(\alpha'_1-\beta'_1) & \alpha_1 > \beta_1 \\
-p \pi^{\mu} \Gamma_p(\alpha'_1-\beta'_1 + 1)& \alpha_1 < \beta_1. \\
\end{cases}
\end{equation}
Now note that if $\alpha'_1 - \beta'_1$ and $\alpha_1 - \beta_1$ are of opposite sign,
we cannot have $\mu = 0$, and so we can rewrite $(\alpha'_1 - \beta'_1) \Gamma_p(\alpha'_1 - \beta'_1)$ as $-\Gamma_p(\alpha'_1 - \beta'_1 + 1)$ or vice versa. This yields the stated formula.
\end{proof}

\subsection{Specialization and factorization}
\label{subsec:factorization}

Using the GKZ interpretation, we may immediately extend the previous computation to arbitrary rank.

\begin{hypothesis} \label{H:factorization}
Throughout \S\ref{subsec:factorization}, suppose that $m \leq n$; $\alpha_i, \beta_j \in \ZZ_{(p)} \cap [0,1)$
for $i,j=1,\dots,n$; and $\alpha_i \neq \beta_j$ for $i,j=1,\dots,n$.
Define $\alpha'_i := \{p\alpha_i\}, \beta'_j := \{p\beta_j\}$. 
\end{hypothesis}

\begin{theorem} \label{T:initial condition}
Suppose that $k \in \{1,\dots,n\}$ is such that $\beta_j \neq \beta_k$ for $j \neq k$.
Then the matrix $\Phi_\lambda$ for $\lambda = \beta_k$ is the $1 \times 1$ scalar
\[
\prod_{i=1}^m \gamma(\alpha'_i - \beta'_k + 1, \alpha_i - \beta_k + 1) \prod_{j=1}^n \gamma(\beta'_j - \beta'_k + 1, \beta_j - \beta_k + 1)^{-1}
\]
(Note that the factor $j=k$ contributes $1$ to the product.)
\end{theorem}
\begin{proof}
For ease of notation we treat only the case $k=n$.
In this case, under the GKZ interpretation, we may read off $\Phi_\lambda$ by specializing $x_{m+n}$ to 0
via the morphism from Construction~\ref{const:specialization}.
In this case, as per Remark~\ref{R:drop last column} we obtain the specified factorization.
\end{proof}

By combining Theorem~\ref{T:initial condition} with Corollary~\ref{C:rank 1 Frob}, we get an explicit formula for the initial condition for the Frobenius intertwiner in the case where $\beta_1,\dots,\beta_n$ are pairwise distinct mod $\ZZ$.
\begin{cor} \label{C:full rank Frob}
In addition to Hypothesis~\ref{H:factorization}, suppose that $\beta_1,\dots,\beta_n$ are pairwise distinct.
Consider the formal solution matrix obtained by multiplying the function
\eqref{eq:formal hypergeom} corresponding to $\beta_k$ by the scalar factor
\begin{equation} \label{eq:rescale factor}
\frac{\prod_{i=1}^m (\alpha_i - \beta_k)_+}{\prod_{j=1}^n (\beta_j-\beta_k)_+}, \qquad
(x)_+ := \begin{cases} x & x > 0 \\
1 & x \leq 0.
\end{cases}
\end{equation}
Define the \emph{zigzag function} associated to $\underline{\alpha}, \underline{\beta}$ as the function $Z: \RR \to \RR$ given by
\[
Z(x) = \# \{i \in \{1,\dots,m\}: \alpha_i < x\} - 
\# \{j \in \{1,\dots,n\}: \beta_i < x \}.
\]
Then the sole entry of $\Phi_\lambda$ for $\lambda = \beta_k$ can be written as
\[
(-1)^{Z(\beta'_k)} p^{Z(\beta_k)} \mu^c
\frac{\prod_{i=1}^m \Gamma_p(\{\alpha'_i - \beta'_k\})}{\prod_{j=1}^n \Gamma_p(\{\beta'_j -\beta'_k\})}
\]
for
\[
c := \sum_{i=1}^n (p\alpha_i- \alpha'_i) - \sum_{j=1}^n (p\beta_j - \beta'_j).
\]
\end{cor}

\begin{remark} \label{R:mu factor}
In Corollary~\ref{C:full rank Frob}, the factor $\mu^c$ does not depend on $k$. We may thus eliminate it at the expense that our normalization no longer matches that of Theorem~\ref{T:DGK formula}.
\end{remark}

\begin{remark}
In applications, we will typically be interested in the case where, in addition to the conditions of
Hypothesis~\ref{H:factorization}, one has that $m=n$ and $\underline{\alpha}, \underline{\beta}\subset \ZZ_{(p)} \cap [0,1)$ are \emph{Galois-stable}, meaning that any two elements of $\ZZ_{(p)} \cap [0,1)$ with the same denominator occur with the same multiplicity in $\underline{\alpha}$ and $\underline{\beta}$. These conditions ensure the existence of a family of \emph{hypergeometric motives} with this hypergeometric equation as associated Picard--Fuchs equation.

In this situation, a further renormalization beyond that of Remark~\ref{R:mu factor} is sometimes warranted
in order to ensure that the Frobenius structure correctly computes the characteristic polynomials of the $p$-Frobenius of the associated hypergeometric motives.
This is achieved by taking the entry of $\Phi_\lambda$ to be
\[
(-1)^{Z(\beta'_k)} p^{Z(\beta_k)-\min\{Z(\beta_*)\}} 
\frac{\prod_{i=1}^m \Gamma_p(\{\alpha'_i - \beta'_k\})/\Gamma_p(\alpha'_i)}{\prod_{j=1}^n \Gamma_p(\{\beta'_j -\beta'_k\})/\Gamma_p(\beta'_j)}.
\]
The net effect of the factors $\Gamma_p(\alpha'_i)$ and $\Gamma_p(\beta'_j)$ is limited by the identity
\[
\Gamma_p(x) \Gamma_p(1-x) = (-1)^y, \qquad y \in \{1,\dots,p\}, \quad y \equiv x \pmod{p}
\]
and its special case
\[
\Gamma_p\left( \frac{1}{2} \right)^2 = \genfrac(){}{}{-1}{p} \qquad (p \neq 2).
\]
\end{remark}

\subsection{An example with repeated parameters}

In lieu of extending Theorem~\ref{T:initial condition} to the case where the $\beta_j$ are not all distinct (which would create some notational headaches), we sketch an example originally due to Shapiro \cite{shapiro1, shapiro2}.

\begin{example} \label{exa:Dwork threefolds}
Consider the case
\[
m=n=4, \qquad \underline{\alpha}; \underline{\beta} = \left( \frac{1}{5}, \frac{2}{5}, \frac{3}{5}, \frac{4}{5} \right); (1,1,1,1).
\]
This example is well-known; the corresponding hypergeometric equation is a Picard--Fuchs equation for the Dwork pencil of quintic threefolds.

Assume $p \neq 2, 5$. (The restriction $p \neq 5$ is essential; the restriction $p=2$ is probably not, but
is made in \cite{shapiro2}.)
For $\lambda = 0$, the matrix $\Phi_{\lambda,0}$ is upper-triangular with eigenvalues $1,p,p^2,p^3$. To compute the off-diagonal entries, we use $p$-adic interpolation: consider the statement of Corollary~\ref{C:full rank Frob} for
\[
\underline{\beta} = ( 1, 1+\epsilon, 1+2\epsilon, 1+3\epsilon), \qquad \epsilon := \frac{p^n}{3p^n+1}.
\]
For $U$ the formal solution matrix, the matrix $\Phi_U$ equals the diagonal matrix 
whose $k$-diagonal entry (for $k=0,\dots,3)$ equals
\[
z^{(p-1)k\epsilon} (-p)^k \frac{\Gamma_p(-1/5 - k \epsilon) \Gamma_p(-2/5 - k\epsilon) \Gamma_p(-3/5 - k \epsilon) \Gamma_p(-4/5 - k \epsilon)}{\Gamma_p(-k\epsilon) \Gamma_p((1-k)\epsilon) \Gamma_p((2-k)\epsilon) \Gamma_p((3-k)\epsilon)}.
\]
Using Definition~\ref{D:Dwork definition of gamma} and Theorem~\ref{T:DGK formula}, one may compute coefficients of the Taylor series for $\Gamma_p$ (see for example \cite[Proposition~3.1]{shapiro2});
we may thus rewrite $\Phi_0$ truncated at $\epsilon^4$, and the formal solution matrix $U$ truncated at $\epsilon^4$ and $z^4$.
Taking the limit of $\Phi = U \Phi_U \sigma(U^{-1})$ as $\epsilon \to 0^+$, 
and using the relationship between derivatives of $\Gamma_p$ and $p$-adic zeta values
(e.g., see \cite[Proposition~11.5.19]{cohen2}),
one may recover Shapiro's formula
\[
\Phi_\lambda = \begin{pmatrix}
1 & 0 & 0 & \frac{2^3}{5^2}(p^3-1) \zeta_p(3) \\
0 & p & 0 & 0 \\
0 & 0 & p^2 & 0 \\
0 & 0 & 0 & p^3
\end{pmatrix}.
\]
We leave further details to the interested reader.
\end{example}

\section{Applications to computation of $L$-functions}

The formula of Dwork can be used as part of an efficient algorithm for computing Euler factors of $L$-functions associated to hypergeometric motives. We sketch this here. (In the case $n=2$,
an alternate approach has been described by Asakura \cite{asakura}.)

\subsection{Hypergeometric motives}

\begin{defn}
Suppose that $m=n$ and that $\underline{\alpha}, \underline{\beta} \subset \QQ$ are both Galois-stable.
Then there exists a family of motives $H(\underline{\alpha}; \underline{\beta}; t)$ over $\QQ(t)$
which for $t \neq \{0,1\}$ is pure of dimension $n$ and weight
\[
w = \max(Z) - \min(Z) - 1
\]
where $Z$ denotes the zigzag function defined in Corollary~\ref{C:full rank Frob}. For example, this motive can be found inside the family of varieties considered in \cite{beukers-cohen-mellit}.

If we specialize to a value of $t$ in $\overline{\QQ}$, then the motive $H(\underline{\alpha}; \underline{\beta}; t)$ has good reduction at all places of the number field $\QQ(t)$ at which $\alpha_1,\dots,\alpha_n,\beta_1,\dots,\beta_n$ have nonnegative valuation and $t, t^{-1}, t-1$ have nonnegative valuation.
An excluded prime is said to be \emph{wild} if the first condition fails (note that this does not depend on $t$)
and \emph{tame} otherwise.
\end{defn}

\begin{remark}  \label{R:CM motive}
When the Galois-stable condition holds and the $\beta_j$ are pairwise distinct, the specialization of
$H(\underline{\alpha}; \underline{\beta}; t)$ at $t=0$ is a CM motive, whose associated $L$-function is therefore given by certain Jacobi sums. The formula given in Corollary~\ref{C:full rank Frob} can also be derived by applying the Gross-Koblitz formula to these Jacobi sums.

When the $\beta_j$ are not pairwise distinct, the specialization of $H(\underline{\alpha}; \underline{\beta}; t)$ at $t=0$ becomes a mixed motive, whose $L$-function then includes a contribution from extension classes. 
Again, it should be possible to make an explicit link with degenerations of Corollary~\ref{C:full rank Frob}; for example, in Example~\ref{exa:Dwork threefolds}, the appearance of $\zeta_p(3)$ should be related via motivic considerations to a corresponding appearance of $\zeta(3)$ in mirror symmetry \cite{kim-yang}.
\end{remark}

\begin{remark}
We expect that there are corresponding families of motives associated to GKZ systems associated to parameters $(A,\delta)$ with $\delta \in \QQ^d$, under a suitable analogue of the Galois-stable condition:
for each prime $p$ for which $\delta \in \ZZ^d_{(p)}$, the GKZ system with parameters $(A, p\delta)$ should be isomorphic to the original one.
\end{remark}

\begin{remark} \label{R:swap alpha beta}
The families $H(\underline{\alpha}; \underline{\beta}; t)$ and $H(\underline{\beta}; \underline{\alpha}; t^{-1})$ are isomorphic. This can be used in certain cases where one wants to make an asymmetric restriction on $\underline{\alpha}$ and $\underline{\beta}$, as in our computation of Frobenius structures.
\end{remark}

\subsection{The approach via trace formulas}

Before describing the approach we have in mind, we describe an alternate approach for purposes of comparison.

\begin{remark}
Suppose that $t \in \QQ \setminus \{0,1\}$ and $p$ is a prime at which $H(\underline{\alpha}; \underline{\beta}; t)$ has good reduction. For $f$ a positive integer, let $H_{p^f}$ be the trace of the $f$-th power of the $p$-Frobenius acting on $H(\underline{\alpha}; \underline{\beta}; t)$; note that this depends only on the residue of $t$ modulo $p$.
By combining \cite{beukers-cohen-mellit} with the Gross-Koblitz formula, one can obtain a highly practical formula
for $H_{p^f}$; this is a poorly documented result of Cohen--Rodriguez Villegas--Watkins, but the formula can be found in the documentation of the \magma\ package on hypergeometric motives:
\begin{center}
\url{http://magma.maths.usyd.edu.au/magma/handbook/hypergeometric_motives}.
\end{center}
The same formula is also implemented in \sage.
\end{remark}

\subsection{The approach via Frobenius structures}

To simplify this discussion, we assume that $\beta_1,\dots,\beta_n$ are pairwise distinct. Recall that via
Remark~\ref{R:swap alpha beta}, we can swap $\underline{\alpha}$ with $\underline{\beta}$ to achieve these conditions in some cases where it is not initially satisfied.

Let $N$ denote the companion matrix for the differential operator $P(\underline{\alpha}, \underline{\beta})$.
Let $U$ denote the formal solution matrix obtained from the matrix $U$ of Corollary~\ref{C:formal solution matrix}
by multiplying its $k$-th column by the factor \eqref{eq:rescale factor} for $k=1,\dots,n$. By Theorem~\ref{T:DGK formula} and Corollary~\ref{C:full rank Frob}, there is a Frobenius structure on $N$ with $\Phi = \Phi_0 \sigma(U^{-1})$, where $\Phi_0$ is the matrix with 
\[
(\Phi_0)_{i,j} = (-1)^{Z(\beta_i)} p^{Z(\beta_j)-\min\{Z(\beta_*)\}} 
\frac{\prod_{k=1}^n \Gamma_p(\{\alpha_k - \beta_i\})/\Gamma_p(\alpha_k)}{\prod_{k=1}^n \Gamma_p(\{\beta_k -\beta_i\})/\Gamma_p(\beta_k)} t^{1-p+\lfloor p \beta_j \rfloor}
\]
whenever $\beta_i \equiv p \beta_j \pmod{\ZZ}$ and $(\Phi_0)_{i,j} = 0$ otherwise.
Note that this computation nominally takes place in $\QQ_p ((t))$; in order to represent the elements of $\Phi$ as rigid analytic functions, we must multiply by a suitable power of $t-1$, then truncate modulo suitable powers of $p$ and $t$. One can then specialize $t$ to any $(p-1)$-st root of unity to obtain a matrix whose characteristic polynomial gives the Euler factor of $H(\underline{\alpha}; \underline{\beta}; p)$. (Beware that we have not yet checked that the scalar normalization is correct. One way to do this would be to use this formula to reprove the Beukers--Cohen--Mellit trace formula.)

We have an experimental \sage\ implementation of this algorithm, and have done numerous tests to confirm its agreement with Beukers--Cohen--Mellit (albeit without fixing the precision estimates; see below). 
See \cite{github}.

\begin{remark} \label{R:precision}
In order to make the previous algorithm rigorous, one must bound the $p$-adic and $t$-adic precision requirements.
The power of $t-1$ can be estimated using the method of \cite{kedlaya-tuitman}. This depends on estimating the $p$-adic valuation of $\Phi_0$; this appears to be controlled by the $p$-adic valuations of the differences $\alpha_k - \beta_i$ and $\beta_k - \beta_i$. In any case, it appears that for a fixed $p$-adic truncation (which suffices for the computation of Euler factors), the power of $t-1$ is bounded independently of $p$; this means that the $t$-adic truncation can be bounded by $cp$ for some constant $c$ independently of $p$.

This has the following consequences for an average polynomial time algorithm. One is trying to evaluate the entries of the matrix
\[
(t-1)^e U \Phi_0 \sigma(U^{-1}),
\]
modulo some fixed power of $p$; they look like polynomial of degree bounded by $cp$ where $c$ is independent of $p$. This means that for the purposes of evaluation $\sigma(U^{-1})$, we need only a constant number of terms of $U^{-1}$; these coefficients are moreover rational numbers with no dependence at all on $p$.
We may thus frame the problem as that of computing, for various primes $p$, a certain $\QQ$-linear combination of coefficients of terms of $U$ of the form $t^{ap+b}$ for certain fixed pairs $(a,b)$, then reducing the result modulo a fixed power of $p$.
\end{remark}

\subsection{Comparison of approaches}

When comparing the relative efficacy of the trace formula and Frobenius structures, it is important to separate different use cases in which the relative strengths of the approaches play different roles. In the following discussion, we mostly ignore constants and logarithmic factors.

\begin{remark}
Suppose we wish to compute $H_p(\underline{\alpha}; \underline{\beta}; t)$ for a single choice of $\underline{\alpha}, \underline{\beta}, t$ and a single prime $p$. Both approaches have complexity linear in $p$; however, the trace formula carries less overhead in this context and thus is preferable in practice. Moreover, if one repeats the computation for the same $p$ and different values of $\underline{\alpha}, \underline{\beta}$, one can cache the Mahler expansion of $\Gamma_p$ for additional savings.
\end{remark}

\begin{remark}
Suppose we fix $\underline{\alpha}, \underline{\beta}, p$,
and wish to compute either $H_p(\underline{\alpha}; \underline{\beta}; t)$ for all values of $t$ (or equivalently, for $t \in \{2,\dots,p-1\}$). In this case, the trace formula can be computed as a polynomial in a variable (running over $(p-1)$-st roots of unity);
alternatively, the Frobenius structure can be computed and then specialized repeatedly.
\end{remark}

\begin{remark}
Suppose we wish to compute the full Euler factor of the $L$-function associated to $H(\underline{\alpha}; \underline{\beta}; t)$ at a prime
$p$. In this case, the trace formula approach requires computing $H_{p^f}(\underline{\alpha}; \underline{\beta}; t)$ for 
$f$ ranging from 1 to half the degree of the associated $L$-function; the formula is a sum over $p^f-1$ terms.
By contrast, the Frobenius structure computation gives the entire Euler factor at once, with complexity linearly in $p$.
\end{remark}

\begin{remark}
Suppose we wish to compute the first $X$ Dirichlet coefficients of the $L$-function associated to $H(\underline{\alpha}; \underline{\beta}; t)$; this is the relevant use case when making numerical computations with the $L$-function.
 Using the trace formula directly scales quadratically in $X$; however, it should be possible to develop an 
\emph{average polynomial time} algorithm in the sense of Harvey \cite{harvey1, harvey2}
(see also \cite{harvey-sutherland, harvey-sutherland2}).
A partial result has been given by Costa--Kedlaya--Roe \cite{ckr}, who compute $H_p(\underline{\alpha}; \underline{\beta}; t) \pmod{p}$
for all primes $p \leq X$ with complexity linear in $X$. As remarked upon in \cite{ckr}, it should be possible to adapt this approach to compute $H_p(\underline{\alpha}; \underline{\beta}; t)$ exactly for all primes $p \leq X$ with similar complexity.

It is less clear how to include higher prime powers into this approach. However, one can use Frobenius structures to circumvent this difficulty, by directly computing full Euler factors for all primes $p \leq X^{1/2}$, then using these to recover
$H_{p^f}(\underline{\alpha}; \underline{\beta}; t)$ for all prime powers $p^f \leq X$ with $f > 1$.
Since this involves $O(X^{1/2})$ computations each of complexity linear in $X^{1/2}$, this does not dominate the computation of prime
Dirichlet coefficients.
\end{remark}

\begin{remark}
Suppose we wish to compute the full Euler factors of the $L$-function associated to $H(\underline{\alpha}; \underline{\beta}; t)$
at all primes $p \leq X$; this is the relevant use case when studying statistical properties of the Euler factors (e.g., the generalized
Sato-Tate conjecture). In this case, it should be possible (and relatively straightforward) to give an average polynomial time
computation using Frobenius structures; however, we do not develop this point further here.
\end{remark}

\section{Towards $A$-hypergeometric motives}

In this paper, we have used only a restricted form of the theory of $A$-hypergeometric systems. However,
it is likely that the circle of ideas giving rise to hypergeometric motives and $L$-functions can be extended beyond this special case; this question was posed at the end of the introduction of \cite{ckr}. We record here some references that point towards such an extension.

The story of hypergeometric motives begins with Greene's construction of finite hypergeometric sums
\cite{greene}. This was generalized to $A$-hypergeometric systems by Gelfand--Graev \cite{gelfand-graev}.

Greene's sums were reinterpreted in terms of $\ell$-adic cohomology by Katz \cite{katz-esde}. This interpretation was extended to the Gelfand--Graev construction by Lei Fu \cite{fu1}.
The $p$-adic construction described in this paper has been generalized by Fu--Wan--Zhang \cite{fu2}.
However,
we know of no analogue of the Beukers--Cohen--Mellit construction.

\end{document}